\documentclass[12pt,twoside]{amsart}
\usepackage{amsmath}
\usepackage{amssymb}
\usepackage{amscd}
\usepackage{xypic}
\xyoption{all}
\title[A note on period domains of Hodge structures]{A note on period domains of Hodge structures}

\author{Mohammad Reza Rahmati}
\address{Abdus Salam School of Mathematical Sciences , GCU, Lahore, PAKISTAN}
\email{mrahmati@cimat.mx}

\date{07, 04, 2019}

\newcommand{\comments}[1]{}

\newtheorem{theorem}{Theorem}[section]
\newtheorem{proposition}[theorem]{Proposition}
\newtheorem{corollary}[theorem]{Corollary}

\newtheorem{example}[theorem]{Example}



\begin{document}

\begin{abstract}
We present period domains of complex Hodge structures and some of their basic properties from representation theory point of view. Specifically we introduce the automorphic cohomology of period domains and the complex structure on symmetric spaces.
\end{abstract}

\maketitle


\section*{Introduction and Preliminaries}

\vspace{0.2cm}
  
A polarized Hodge structure of weight $n$ on a $\mathbb{Q}$ vector space $V$ together with a $(-1)^n$-symmetric non-degenerate $\mathbb{Q}$-bilinear form $Q : V \times V \to \mathbb{Q}$, is given by a representation  

\begin{equation}
\phi_0: \mathbb{U}(\mathbb{R}) \to Aut(V_{\mathbb{R}},Q) 
\end{equation}

\noindent
defined over $\mathbb{R}$, such that if we denote the $t^p\bar{t}^q$-eigenspace of $\phi(t)$ by $V^{p,q}$, then $p+q=n$ for all non-zero $V^{p,q}$. Here $\mathbb{U}(\mathbb{R})$ is the circle $S^1 \subset \mathbb{C}$ considered over $\mathbb{R}$. We denote $\textbf{C}:=\phi_0(i)$. The adjoint action of the group $G_{\mathbb{R}}=Aut(V_{\mathbb{R}},Q)$ on $\phi$ defines the period domain of polarized Hodge structures of weight $n$ on $V$, denoted by

\begin{equation}
D=\{\phi:S^1 \to G_{\mathbb{R}} \ ; \ \phi=g^{-1}\phi_0 g , \ g \in G_{\mathbb{R}} \}.
\end{equation}

\noindent
Denote the centralizer of $\phi_0$ by $M=Z_{\phi_0}(G_{\mathbb{R}})$, then $D=G_{\mathbb{R}}/M$.

Alternatively a (${Q}$-) polarized $\mathbb{Q}$-Hodge structure on $V$ of weight $n$ is given by a decreasing filtration $F^{\bullet}: F^0=V \supset F^1 \supset ... \supset F^n=0$ on $V_{\mathbb{C}}$ such that the Hodge-Riemann bilinear relations hold

\begin{equation}
Q(F^p, F^{n-p+1})=0, \qquad Q(x,\text{\textbf{C}}.x) > 0 , \ (x \in V).
\end{equation}

\noindent
By setting $V^{p,q}=F^p \cap \overline{F^q}$ we obtain the decomposition $V=\bigoplus_{p+q=n} V^{p,q}$, where $\overline{V^{p,q}}=V^{q,p}$. We denote the set of all flags $F^{\bullet}$ satisfying the first condition in (3) by $\check{D}$. Then one has $\check{D}=G_{\mathbb{C}}/P$ where $P$ is a parabolic subgroup of $G$. 

The Lie algebra $\mathfrak{g}$ of the Lie group $G_{\mathbb{C}}$ is a $\mathbb{Q}$-linear subspace of $End(V)$, and inherits a Hodge structure of weight zero namely

\begin{equation}
\text{Ad}\phi:\mathbb{U}(\mathbb{R}) \to \text{Aut}(\mathfrak{g}_{\mathbb{R}},B)
\end{equation}

\noindent
which is polarized by the Cartan-Killing form $B: \mathfrak{g} \times \mathfrak{g} \to \mathbb{C}$. 

A variation of Hodge structure $\mathcal{V}$ on a quasi-projective variety $S$ can be given by its period map $\Pi:S \to \Gamma_{\mathbb{Z}} \backslash D$, where $\Gamma_{\mathbb{Z}}$ is a discrete subgroup of $G$. One has natural local systems $\mathcal{V}:=\Gamma \backslash (D \times {V}),  \ \mathcal{G}:=\Gamma \backslash (D \times \mathfrak{g})$ on $D$. The stalk of these local systems are the Hodge structures $V$ and $\mathfrak{g}=\mathfrak{gl}(V)$. The data of these local system can be encoded in the monodromy representation of $\rho:\Gamma_{\mathbb{Z}} \to G_{\mathbb{C}}$. We sometimes denote the  local system as $\mathcal{V}^{\rho}$. In fact equivalent representations of the monodromy group $\Gamma_{\mathbb{Z}}$ gives rise to isomorphic local system. Similar criteria holds for the associated vector bundles.

Associated to each nilpotent transformation $N \in \mathfrak{g}$ one defines a limit mixed Hodge structure. The local system $\mathcal{G} \to \Delta^*$ is then equipped with the monodromy $T=e^{\text{ad} N}$ and Hodge filtration defined with respect to the multi-valued basis of $\mathcal{G}$ by $e^{\log(t)\frac{N}{2\pi i}}F^{\bullet}$. They define a limit MHS $(\mathcal{G},F^{\bullet},W(N)_{\bullet})$. One has the hard Lefschetz isomorphism $N^k:Gr_{-k}^{W(N)}\mathcal{G} \stackrel{\cong}{\longrightarrow} Gr_{k}^{W(N)}\mathcal{G}$ and the polarizing form induces nondegenerate pairings 

\begin{equation}
B_k:PGr_k^{W(N)} \mathcal{G} \times PGr_{-k}^{W(N)} \mathcal{G} \to \mathbb{Q}, \qquad B_k(u,v)=B(u,N^kv).
\end{equation}

\begin{example} A basic example of this is the VHS $\mathcal{V}=\mathcal{V}^{1,0} \oplus \mathcal{V}^{0,1}$ of weight $1$, obtained from the middle cohomology of a fibration of curves of genus $g$. In this case 

\begin{equation}
D=\textbf{H}_g=\{Z \in M_{g \times g} : Z=^tZ, \ Im (Z) >0\}
\end{equation}

\noindent
is the Siegel generalized upper half space. The case $g=1$ corresponds to elliptic fibrations. In this case the spaces of differential forms are modular forms on the upper half plane $H$ and the polarization is given by the Peterson inner product; 

\begin{equation}
\langle f,g \rangle =\int_{\Gamma \backslash H}f(z)\overline{g(z)}y^{n-2}dxdy, \qquad (x=Re(z),\ y=Im(z)).
\end{equation}

\noindent
The inner product (7) is $SL_2(\mathbb{Z})$-invariant, cf. \cite{S}.

Let $\mathcal{A}_D^{p,q}(\mathfrak{g})$ be the vector space of smooth $C^{\infty}$-forms on $D$ with values in $\mathfrak{g}=\mathfrak{gl}(\mathcal{V})$. A vector valued $\mathcal{A}^{p,q}$-automorphic form $f \in  \mathcal{A}^{p,q}(D, \mathfrak{g})$ satisfies the automorphy relation 

\begin{equation}
Ad (\rho)f(z)=f(\gamma .z)\gamma'(z)^p \overline{\gamma'(z)^q}, \qquad  (\gamma \in \Gamma) 
\end{equation}

\noindent
where $\rho:\pi_1 \to Gl(V)$ is the monodromy representation.  The hyperbolic metric descended from $H$ defines the Hodge $*$-operator on $\mathcal{V}$ and $\mathcal{G}$ with the Hodge inner product  

\begin{equation}
\langle f,g \rangle =\int_H f(z) \wedge [*g(z)]dxdy=\int_{\Gamma \backslash H} tr(f(z) \wedge \overline{g(z)^{\text{tr}}})y^{p+q-2}dxdy.
\end{equation}

\noindent
In the case of the ordinary automorphic functions the right hand side is the Peterson inner product (2). The polarization form of a Hodge structure is unique up to scaling. Therefore, Peterson inner product must induce the Cartan-Killing form $B$ in the above construction. The primitive subspaces $P_l: \ker(N^{l+1}): Gr_{l}^W \mathcal{S} \to Gr_{-l-2}^W \mathcal{S}$ have pure Hodge structure of weight $k+l$ which are polarized by the forms $B_k$, \cite{TZ}.  
\end{example}

\vspace{0.3cm}

\section{cohomology of period domains}

\vspace{0.3cm}

Assume $G_{\mathbb{R}}$ is a reductive Lie group and $\mathfrak{g}=\text{Lie}(G_{\mathbb{C}})$ is its Lie algebra. Let $T$ be a maximal torus with $\mathfrak{t}=Lie(T)_{\mathbb{R}}$, and $\mathfrak{h}=\mathfrak{t}_{\mathbb{R}} \otimes \mathbb{C}$ a Cartan subalgebra. Fix a set of roots $\Phi=\Phi^+ \cup \Phi^-$ be a system of roots for $G_{\mathbb{C}}$ and $W$ its Weyl group. Write $\mathfrak{g}_{\mathbb{C}}=\mathfrak{t}_{\mathbb{C}} \oplus \mathfrak{n} \oplus \mathfrak{n}^-$ , where $\mathfrak{n}=\oplus_{\alpha \in \Phi^+} \mathfrak{g}^{\alpha}$. Let $\mathfrak{b}=\mathfrak{n} \oplus \mathfrak{h}$ be the associated Borel subalgebra with Lie group $B$. One matches these data with the structure of period domain $D=G_{\mathbb{R}}/M$ in the diagram

\begin{equation}
\begin{CD}
G_{\mathbb{R}}/T @>>> G_{\mathbb{C}}/B\\
@VVV      @VVV\\
D=G_{\mathbb{R}}/M @>>> G_{\mathbb{C}}/P=\check{D}.
\end{CD}
\end{equation}

\noindent
Set $\rho=\sum_{\alpha \in \Phi^+} \alpha \time 1/2$. Identify $T_s^{0,1}=\mathfrak{n}$. Then $\Omega_{D,s}^1=\mathfrak{n}^{\vee}$. Let $L_{\alpha}$ be the line bundle given by the root $\alpha \in \Phi^+$. Then, via the above identification we have 

\begin{equation}
A^{0,q}(D, L_{\alpha})=(C^{\infty}(G_{\mathbb{R}}) \otimes \Lambda^q \mathfrak{n}^{\vee} \otimes L_{\alpha})^T=:(C^{\infty}(G_{\mathbb{R}}) \otimes \Lambda^q \mathfrak{n}^*)_{-\alpha}.
\end{equation}

\noindent
Thus we have the following equality in the cohomological complexes

\begin{equation}
A^{0,q}((D, L_{\alpha}), \bar{\partial})=(C^{\infty}(G_{\mathbb{R}}) \otimes \Lambda^q \mathfrak{n}^{\vee}),\delta)_{-\alpha}
\end{equation}

\noindent
where $\bar{\partial}$ and $\delta$ are differentials. Therefore 

\begin{equation}
H^q(D, L_{\alpha})=H^q(\mathfrak{n},C^{\infty}(G_{\mathbb{R}}))_{-\alpha}.
\end{equation}

\begin{theorem} \cite{GS} If $\alpha + \rho$ is singular, then $H^k(D,L_{\alpha})=0$ for all $k$. If $\alpha + \rho$ is nonsingular, then let $w \in W$ be the element that carries $\alpha +\rho $ into the highest Weyl chamber. Let 

\begin{equation}
l = \sharp \{\alpha \in \Delta_+ \ |\ \omega(\alpha) \text{ is negative }\}.
\end{equation}

\noindent
Then $H^k(D,L_{\alpha})= 0$  when $k \ne l$. If $k=l$ then  $H^k(D,L_{\alpha})$ is an irreducible highest weight representation of $G$ of highest weight $w(\alpha+\rho)-\rho$. 
\end{theorem}

\section{Complex structure on period domains}

\vspace{0.3cm}

Let $G$ be a complex reductive Lie group defined over $\mathbb{R}$ and $\mathfrak{g}:=Lie(G)=\mathfrak{h} \oplus \bigoplus_{\alpha \in R} \mathfrak{g}_{\alpha}$ its root decomposition, and $T$ a maximal torus in $G_{\mathbb{C}}$. Let $M$ be a closed subgroup of $G$ with $\mathfrak{m}=Lie(M_{\mathbb{C}})=\mathfrak{h} \oplus \bigoplus_{\beta \in \Delta} \mathfrak{g}_{\beta}$. The tangent space of $G/M$ at the point $eM$ can be identified with $\mathfrak{g}/\mathfrak{m}$. An $M$-invariant almost complex structure on $G/M$ is given by a linear map $J:\dfrac{\mathfrak{g}}{\mathfrak{m}}(\mathbb{R}) \to  \dfrac{\mathfrak{g}}{\mathfrak{m}}(\mathbb{R})$ such that $J^2=-id$. Assume $\tau$ is the isotropy representation on $M$. Then  

\begin{align}
J(\tau(m).X)=\tau(m).J(X), \qquad m \in \mathfrak{m} , \ X \in \mathfrak{g}. 
\end{align}

\noindent
The restriction of the isotropy representation $\tau$ to $T$ is the adjoint representation. It follows that $J$ leaves each $\mathfrak{g}_{\beta}$ invariant, and reduces to multiplication by $\pm i$ on $\mathfrak{g}_{\beta}, \ (\beta \in R \setminus \Delta)$. Set

\begin{equation}
\textbf{C}:=\{\alpha \in R \setminus \Delta ; \  J|_{\mathfrak{g_{\alpha}}}=i.\  id\ \}.
\end{equation}

\noindent
A choice of $G$-invariant almost complex structure on $G/M$ gives rise to determine $\textbf{C}$ such that if $\alpha \in R \setminus \Delta$, then only one of $\alpha $ or $-\alpha$ is in $\textbf{C}$. It is possible to choose a system of positive roots $R_+$ such that $C \subset R_+$. In this case $R_+$ is called the system of positive roots adopted to invariant almost complex structure. This criteria generalizes to arbitrary homogeneous vector bundles on $G/M$ defined by a representation $\rho: M \to Gl(V)$ on a vector space $V$. In this case we have natural maps $GL(V) \to G \times GL(V)/H \to G/M$, where $H$ is the subgroup of elements of the form $(x, \rho(x)), \ x \in M$.

\begin{proposition}\cite{R}
Let $\rho:M \to Gl(V)$ be an $n$-dimensional complex representation and $E:=G \times_M V$ the associated vector bundle on $G/M$. Then there exists a 1-1 correspondence between $G$-invariant almost complex structures on $E$ and the set

\begin{equation}
Hom_M^J(\ \bigoplus_{\alpha \in C} \mathfrak{g}_{\alpha},\mathfrak{g}'):=\{f :\oplus \mathfrak{g}_{\alpha} \to \mathfrak{g} |\ f \circ J=-J \circ f\} 
\end{equation}

\noindent
where $\mathfrak{g}'=\mathfrak{gl}(V)$, with a choice of $\textbf{C}_{\mathfrak{gl}_n}$ given as (13). Moreover such almost complex structure, will be always a complex (holomorphic) structure. 
\end{proposition} 

The proof uses the split exact sequence of $M$-modules

\begin{equation}
0 \to \mathfrak{g}' \to \mathfrak{p}=\mathfrak{p}=(\mathfrak{g} \times \mathfrak{g}')/\mathfrak{h}_{\mathbb{R}} \to (\mathfrak{g}/\mathfrak{v})_{\mathbb{R}} \to 0.
\end{equation}

\noindent
Then $G$-invariant holomorphic structures correspond to all possible extensions of $ad \rho:\mathfrak{m} \to \mathfrak{gl}(V)$ to $ad \rho:\mathfrak{m} \bigoplus_{\alpha \in \textbf{C}} \mathfrak{g}_{-\alpha} \to \mathfrak{gl}(H)$, see \cite{R} for the proof. 

\vspace{0.3cm}

\section{Geometry of Hodge bundles}

\vspace{0.3cm}

Assume $X$ is a compact connected Riemann surface of genus $g$ and $\emptyset \ne S \subset X$ a divisor on $X$. A parabolic bundle $F$ with parabolic divisor $S$ consists of the data of a filtration (quasi-parabolic condition) 

\begin{equation}
F_s=F_{s,1} \supset ... \supset F_{s,l_s} \supset 0, \qquad (\forall s \in S)
\end{equation}

\noindent
and rational numbers $0 \leq \alpha(s) < ... < \alpha_{l_s} < 1$ called parabolic weights. Fix a positive integer $r$. The parabolic degree of $F$ is defined by

\begin{equation}
\text{Par-deg}(F)=\deg(F)+\sum_{s \in S}\sum_i \alpha_{i(s)}(s)\dim(F_{s,i}/F_{s,i+1}).
\end{equation}

\noindent
A Hermitian structure on a parabolic bundle $F$ is a Hermitian metric $\Vert.\Vert$ on $F_{|_{X \diagdown S}}$ with the extra condition that; around each $s \in S$ for any section $\sigma $ of $F$, that is non-zero at $F_{s,i}$, one has $\Vert \sigma(z) \Vert=\phi(z)|z|^{\alpha_i(s)}$, where $\phi$ is positive real valued. If $F$ is (Hermitian) parabolic, then every sub-bundle of $F$ is (Hermitian) parabolic in a natural way. We call a parabolic bundle semi-stable if the factor 
$(\text{par-deg}/\text{rank})$ is non-increasing when passing to sub-bundles. In case this number is strictly decreasing we call the bundle stable. 

Lets consider the local systems of Hodge structures $\mathcal{V}^{\rho}$ and $\mathcal{G}=\mathfrak{gl}(\mathcal{V}^{\rho})$ in the example (0.1). As explained before in this case the associated vector bundle (which we denote by the same notation) are Hermitian symmetric spaces where the Hermitian form is given by the Peterson inner product. The representation $\rho$ is unitary in this case. The period domain can be identified with the upper half plane. In this case the parabolic degree can be computed to be $0$, [cf. \cite{TZ}]. 

\begin{theorem} (Mehta-Seshadri) \cite{TZ} A parabolic bundle $F$ of rank $k$ and parabolic degree $0$ is stable if and only if it is isomorphic to a bundle $F^{\rho}$, where $\rho:\Gamma_{\mathbb{Z}} \to U(k)$ is an irreducible unitary representation of the monodromy group $\Gamma$ which is admissible with respect to the weights and multiplicities.  
\end{theorem}

The representation $\rho$ in Theorem is called admissible with respect to a given set of weights and multiplicities $(\alpha_l^i,k_l^i)$ at $S$ if for each $i$, there exists a generator $s_i$ for the local monodromy at $P_i \in S$, and $U_i \in U(k)$ and $D_i=\exp(2 \pi i .diag[\alpha_1^i,...,\alpha_{r_i}^i])$ such that $\rho(S_i)=U_iD_iU_i^{-1}$, ( $\alpha_l^i=\alpha_l(P_i)$ is repeated $k_l^i=k_l(P_i)$ times ). 

\begin{corollary}
The vector bundle of the local systems $\ \mathcal{V}^{\rho},\  \mathcal{G}=\mathfrak{gl}(\mathcal{V}^{\rho})\ $ in the example (0.1) are stable parabolic vector bundles over $\Gamma \backslash H$.
\end{corollary}

ACKNOWLEDGEMENT: I thank the Abdus Salam School of Mathematical Sciences , GCU, for its research facilities and financial support.


\begin{thebibliography}{99}

\bibitem{GGK} M. Green, P. Griffiths, M. Kerr, Mumford-Tate domains,  Bollettino dell' UMI (9) III (2010), 281-307.

\bibitem{G} P. Griffiths, Hodge theory and representation theory, Ten lectures given at TCU, June 18-22, 2012

\bibitem{GS} P. Griffiths, W. Schmid , Locally homogeneous complex manifolds, 
 
\bibitem{R} S. Ramanan, Homogeneous vector bundles on homogeneous space, Topology Vol 5. , pp. 159-177 , 1966

\bibitem{S} J. P. Serre, A course in Arithmetic, Graduate text in mathematics (7), New York, Springer-Verlag 1973

\bibitem{Sch} W. Schmid, Variation of hodge structure: The singularities of the period mapping, Inventiones mathematicae, Volume 22, Issue 3–4, pp 211–319,  1973
 
\bibitem{TZ} L. A. Takhtajan, P. G. Zograf, The first chern form on moduli of parabolic bundles, Volume 341, Issue 1, pp 113–135, 2008

\end{thebibliography}
\end{document}